\documentclass[10pt,twoside]{article}

\RequirePackage{amsmath,amssymb,amsthm,graphicx,float,xcolor,geometry,cite,enumitem,fancyhdr,lipsum,subfig,sidecap, url}

\newtheorem{theorem}{Theorem}

\makeatletter

\def\vec#1{\ensuremath{\mathchoice
                     {\mbox{\boldmath$\displaystyle\mathbf{#1}$}}
                     {\mbox{\boldmath$\textstyle\mathbf{#1}$}}
                     {\mbox{\boldmath$\scriptstyle\mathbf{#1}$}}
                     {\mbox{\boldmath$\scriptscriptstyle\mathbf{#1}$}}}}%

\setlist[enumerate]{topsep=1ex,itemsep=-0.5ex,leftmargin=*}
\setlist[enumerate,2]{topsep=-0.5ex,itemsep=-0.5ex,leftmargin=*}
\setlist[itemize]{topsep=1ex,itemsep=-0.5ex,leftmargin=*}
\setlist[itemize,2]{topsep=-0.5ex,itemsep=-0.5ex}

\title{Self-foldability of monohedral quadrilateral origami tessellations}
\author{Thomas C. Hull\thanks{Western New England University, {\tt thull@wne.edu}} \and 
Tomohiro Tachi\thanks{University of Tokyo, Japan, {\tt tachi@idea.c.u-tokyo.ac.jp}}}
\date{}

\makeatletter
  \let\runtitle\@title
  \let\runauthor\shortauthor
\makeatother

\usepackage{bm}

{\makeatletter
 \gdef\xxxmark{%
   \expandafter\ifx\csname @mpargs\endcsname\relax 
     \expandafter\ifx\csname @captype\endcsname\relax 
       \marginpar{xxx}
     \else
       xxx 
     \fi
   \else
     xxx 
   \fi}
 \gdef\xxx{\@ifnextchar[\xxx@lab\xxx@nolab}
 \long\gdef\xxx@lab[#1]#2{\textbf{[\xxxmark #2 ---{\sc #1}]}}
 \long\gdef\xxx@nolab#1{\textbf{[\xxxmark #1]}}
}

\def\vec{\mathbf} 

\newcommand{\vecrho}{\bm \rho}

\begin{document}

\maketitle

\begin{abstract}
Using a mathematical model for self-foldability of rigid origami, we determine which monohedral quadrilateral tilings of the plane are uniquely self-foldable.  In particular, the Miura-ori and Chicken Wire patterns are not self-foldable under our definition, but such tilings that are rotationally-symmetric about the midpoints of the tile are uniquely self-foldable.

\end{abstract}

\section{Introduction}
\label{sec:introduction}
In many applications of origami in physics and engineering one wants a sheet of material to fold on its own into some shape, say with some actuators giving a rotational moment of force at the creases.  Such a process is called {\em self-folding}, and the challenge of programming the actuator forces to guarantee that the crease pattern folds into the desired shape is an area of growing research interest \linebreak \cite{Chen2018,Stern2018}.  In 2016 the authors proposed a mathematical definition for self-folding rigid origami along with a model for determining when a set of actuator forces, which we call {\em driving forces}, will be guaranteed to self-fold a given rigidly-foldable crease pattern in a predictable way \cite{Tachi2016}.  However, this model has not yet been tested on very many crease patterns.

In this paper we use the mathematical model for self-foldability to show that two well-known crease patterns, the Miura-ori and the Chicken Wire pattern (see Figure~\ref{TT-fig0}), are not, in fact, uniquely self-foldable.  That is, no set of driving forces will by themselves be guaranteed to fold these patterns in a completely predictable way; rather, there will always be multiple folded states into which a given set of driving forces could fold.  We also argue that it is the high degree of symmetry in the Miura and Chicken Wire patterns that lead to this behavior.  In fact, we prove that crease patterns that form a monohedral tiling from a generic quadrilateral tile, where the tiling is rotationally-symmetric about the midpoint of each side of the quadrilaterals, have a driving force that uniquely self-folds the crease pattern into a desired target shape.

\begin{figure}
	\centering
	\includegraphics[width=\linewidth]{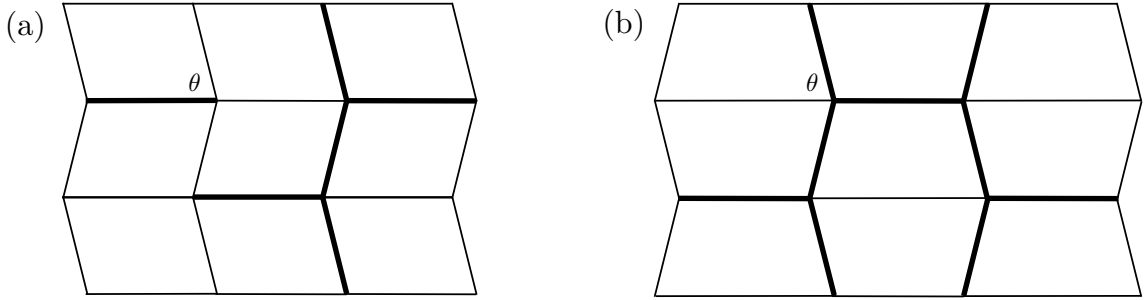}
	\caption{$3\times 3$ sections of the (a) Miura-ori and (b) Chicken Wire crease patterns, generated by a parallelogram or isosceles trapezoid tile with acute angle $\theta$.}
	\label{TT-fig0}
\end{figure}

\section{Self-folding}
\label{sec:section1}
We briefly describe the self-folding mathematical model  from \cite{Tachi2016}.
Given a rigidly-foldable crease pattern with $n$ creases, we define its {\em configuration space} $S\subset\mathbb{R}^n$ to be the set of points $x\in\mathbb{R}^n$ where the $i$th coordinate of $x$ is the folding angle of the $i$th crease in a rigidly-folded state of the crease pattern.  A {\em rigid folding} is defined to be an arc length-parameterized piecewise $C^1$ curve $\vecrho(t)$ in $S$ for $t\in[0,s]$ where $\vecrho(0)$ is the initial state of the folded object and $\vecrho(s)$ is the target state.  For self-folding, we want to find rotational forces to place on the creases that will force the crease pattern to fold from the initial state to the target state along the curve $\vecrho(t)$, and to do that we consider vector fields on our configuration space $S$.  
A vector field $\vec f$ on $S$  is called a {\em driving force} for the rigid folding $\vecrho(t)$ if $\vecrho'(t)\cdot {\vec f}(\vecrho(t))>0$ for all $t\in[0,s]$.  Depending on the specifics of the crease pattern and the chosen folding path $\vecrho(t)$, there may be many different vector fields $\vec f$ that act as driving forces for $\vecrho(t)$.
We also constrain the driving force to be additively separable function, i.e., $ \vec f(\vecrho) = f_1\vec e_1(\rho_1) +\cdots+f_n\vec e_n (\rho_n)$, where $\vec e_i$ represents the unit vector along the $i$th axis of the configuration space\footnote{
The physical interpretation of an additively separable driving force is that each actuator at the crease knows its own folding angle, but not  the others, i.e., actuators are not communicating each other.}.

Define $d(t)=\vecrho'(t)\cdot {\vec f}(\vecrho(t))$, which is called the {\em forward force} of $\vec f$ along $\vecrho(t)$. This quantity measures the cosine of angle between ${\vec f}(\vecrho(t))$ and $\vecrho'(t)$ and thus represents the amount of force ${\vec f}(\vecrho(t))$ contributes to pushing in the direction of $\vecrho'(t)$ along the rigid folding curve $\vecrho(t)$\footnote{This amount corresponds to the power in a physical sense.}.

Note that given a point $x\in S$, there could be many rigid foldings $\vecrho(t)$ in the configuration space $S$ passing through $x$.
We consider the set of tangents of such valid folding passes projected on to an $n$-dimensional unit sphere called {\em valid tangents} $V_x$ .
$V_x$ could be made of disconnected components, and these can be useful in self-foldability analysis.
We say that a rigid folding $\vecrho(t)$ is {\em self-foldable} by a driving force $\vec f$ if the forward force $d(t)$ is a local maximum on $V_{x(t)}$ for $t\in[0,s]$. 
Maximizing $d(t)$ among the tangent vectors insures that $\vec f$ will push the folding in the direction of the tangent $\vecrho'(t)$ and thus along the curve $\vecrho(t)$. 
 If $\vecrho(t)$ is the only rigid folding that is self-foldable by $\vec f$, then we say $\vecrho(t)$ is {\em uniquely self-foldable by $\vec f$}.
When considering the unique self-folding along a folding motion $\vecrho(t)$, we are interested in the ``wrong mode(s)'', which corresponds to the components of $V_x$ that do not contain the projection of $\pm\vecrho'(t)$, which we call the valid tangents \emph{surrounding} $\pm\vecrho'(t)$ and denote it by $\overline{V_x}$.

%

The following theorems, proved in \cite{Tachi2016}, can helpful to determine whether or not a given rigidly-foldable crease pattern is uniquely self-foldable.

\begin{theorem}\label{TT-thm1}
A rigid folding $\vecrho(t)$ from the unfolded state $\vecrho(0)={\vec 0}$ to a target state is uniquely self-foldable by a driving force $\vec f$  only if at the unfolded state $t=0$ we have $\vec f$  is perpendicular to every tangent vector in the valid tangents surrounding $\pm\vecrho'(0)$.
\end{theorem}

\begin{theorem}\label{TT-thm2}
Consider an origami crease pattern at the flat, unfolded state and the tangent  $\vecrho'(0)$ of a rigid folding $\vecrho(t)$ at that state.  Then a driving force $\vec f$ exists to make the rigid folding uniquely self-foldable from the flat state only if the dimension $m$ of the tangent space $T_{\vecrho(0)}S$ (which is the solution space of first-order (infinitesimal) constraints of the rigid origami mechanism) is strictly greater than the dimension $n$ of the linear space spanned by the vectors in the valid tangents surrounding $\pm\vecrho'(0)$.  (I.e., we need $m>n$.)
\end{theorem}

\begin{theorem}\label{TT-thm3}
Given an origami crease pattern made of quadrilateral faces whose vertices form an $a\times b$ grid, the dimension of the tangent space $T_{\vec 0}S$ is $a+b$ (where $\vec 0$ represents the flat, unfolded state).
\end{theorem}

For the purposes of this paper, we introduce an addition to these theorems not previously published.

%
%

\begin{theorem}\label{TT-cor2}
Given a rigidly-foldable origami crease pattern that has a one degree of freedom (1-DOF) rigid folding path $\vecrho(t)$, then a driving force $\vec f$ exists to make $\vecrho(t)$ uniquely self-foldable from the flat state if and only if $\vecrho'(0)$ is not contained in the linear space spanned by the tangent vectors in the valid tangents $\overline{V_{\vecrho(0)}}$ surrounding $\pm\vecrho'(0)$.
\end{theorem}
\begin{proof}
Sufficiency follows from Theorem~\ref{TT-thm1}. 
In order for the driving force $\vec f(t)$ to make $\vecrho(t)$ uniquely self-foldable, $\vec f(0)$ needs to be perpendicular to every tangent vector in $\overline{V_{\vecrho(0)}}$ and thus perpendicular to the linear space containing $\overline{V_{\vecrho(0)}}$.
$\vecrho'(t)$ cannot be contained in the linear space containing $\overline{V_{\vecrho(0)}}$ because otherwise the forward force will be $0$.
Now we consider the orthogonal projection $\vecrho'(0)^\|$ of $\vecrho'(0)$ onto the linear space spanned by $\overline{V_{\vecrho(0)}}$.
Then $\vec f(0)=\vecrho'(0)-\vecrho'(0)^\|$ is perpendicular to the linear space spanned by $\overline{V_{\vecrho(0)}}$, and gives positive dot product with $\vecrho'(0)$, so making it uniquely self-foldable.
\end{proof}

Finally, we include another definition.  Given a rigidly-foldable crease pattern that has 1-DOF as a rigid folding mechanism, its configuration space $S$ will contain curves that pass through the origin (the unfolded state).  These different curves are referred to as the different {\em folding modes} (or just {\em modes}) of the rigid origami.  The modes represent the different ways that the crease pattern can rigidly fold from the flat state.  (By convention, two rigid foldings that have the same folding angles except for switching mountains to valleys and vice-versa are considered part of the same mode.)

\section{Origami tessellations and the main theorem}
\label{sec:section2}

\begin{figure}
	\centering
	\includegraphics[width=\linewidth]{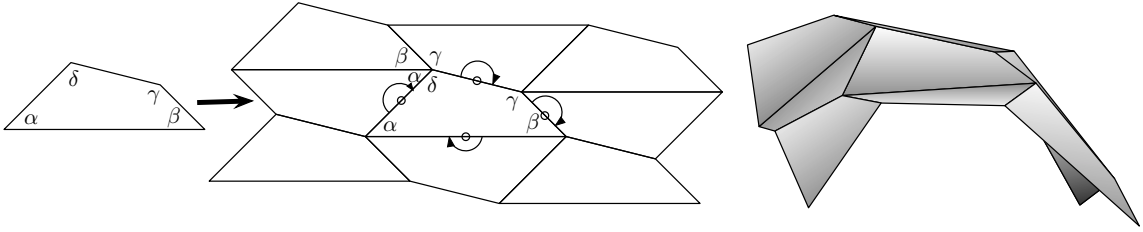}
	\caption{A rotationally-symmetric, non-flat-foldable origami tessellation.}
	\label{TT-fig2}
\end{figure}

Origami tessellations are origami folds whose crease patterns are a tiling (or a subset of a tiling)  of the plane.  We will only be considering origami tessellations that are also rigidly-foldable.  A {\em monohedral} tessellation is one in which all the tiles are the same, so that the tiling is generated by a single tile.  For example, the classic Miura-ori is a rigidly-foldable, monohedral origami tessellation that is generated by a parallelogram.

Another class of such tessellations can be made from any convex quadrilateral by repeatedly rotating about the midpoint of each side of the tile by $180^\circ$.  We will call such a quadrilateral tiling a {\em rotationally-symmetric} tiling; see Figure~\ref{TT-fig2}. Because the interior angles of the tile ($\alpha,\beta,\gamma, \delta$ in Figure~\ref{TT-fig2}) must sum to $360^\circ$, we know that each vertex of such a tiling will be 4-valent.  From this, we know that the tiling will be a crease pattern that is rigidly foldable by some finite amount from the flat state.\footnote{This is because  monohedral quadrilateral tilings fit into the class of {\em Kokotsakis polyhedra}, the rigid flexibility of which has been studied for quite some time \cite{Kokotsakis1933} and recently classified in the quadrilateral case \cite{Izmestiev2017}.} If we also require that the opposite angles of the quadrilateral tile be supplementary, then Kawasaki's Theorem will be satisfied at each vertex of the resulting tessellation, implying that the tessellation will be (locally) flat-foldable; see Figure~\ref{TT-fig1}.

\begin{figure}
	\centering
	\includegraphics[width=\linewidth]{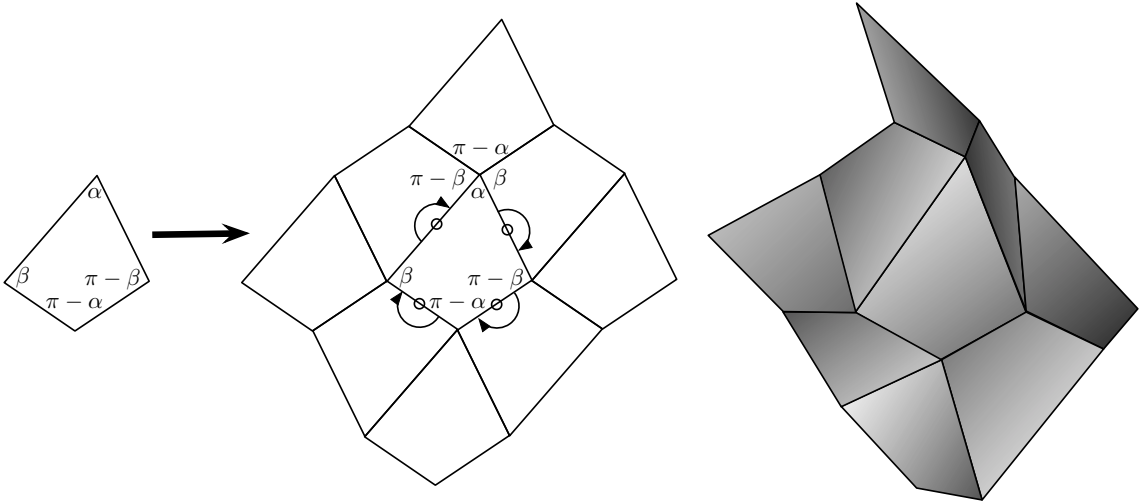}
	\caption{A rotationally-symmetric flat-foldable origami tessellation.}
	\label{TT-fig1}
\end{figure}

The main results of this paper are the following:

\begin{theorem}\label{TT-thm4}
Let $C$ be a monohedral quadrilateral origami tessellation that is rigidly foldable with rigid folding $\vecrho(t)$ for $t\in[0,s]$ for any initial folded state $\vecrho(0)$ and target state $\vecrho(s)$.  Label the interior angles of the quadrilateral tile $\alpha$, $\beta$, $\gamma$, and $\delta$ as in Figure~\ref{TT-fig2}.

(a) If $C$ is the Miura-ori (so the tile is a parallelogram), then $\vecrho(t)$ is not uniquely self-foldable.

(b) If $C$ is rotationally-symmetric and locally flat-foldable with $\alpha=\beta<90^\circ$ (so the tile is an isosceles trapezoid), then $\vecrho(t)$ is not uniquely self-foldable.  (This is the Chicken Wire pattern \cite{Evans:2015}.)

(c) If $C$ is rotationally-symmetric and locally flat-foldable with $\alpha<90^\circ$ and $\alpha\not=\beta$ (as in Figure~\ref{TT-fig1}), then there exists a driving force $\vec f$ that makes $\vecrho(t)$ uniquely self-foldable.  

(d) If $C$ is rotationally-symmetric and not locally flat-foldable with $\alpha<90^\circ$ and $\alpha\not=\beta, \delta$ (as in Figure~\ref{TT-fig2}), then there exists a driving force $\vec f$ that makes $\vecrho(t)$ uniquely self-foldable. 
\end{theorem}

What we see here is that monohedral quadrilateral origami tessellations that have too much symmetry, like the Miura-ori and the Chicken Wire patterns, are not uniquely self-foldable by our definition.  But {\em generic} quadrilateral origami tessellations that are rotationally-symmetric will be uniquely self-foldable.  This gives us a large class of tessellation patterns that are self-foldable.

\section{Proof of the main theorem}

\begin{proof}[Proof of (a) and (b)]

The standard way of rigidly folding the Miura-ori is shown in Figure~\ref{TT-fig3}(a).  When such a rigid origami flexes, the folding angles (that is, $\pi$ minus the dihedral angles) of the creases that meet at a vertex have an especially nice relationship for degree-4, flat-foldable vertices that has become a standard tool in origami engineering (see, for example, \cite{Evans:2015}).  Specifically, the folding angles that are opposite of each other at a vertex will be equal in magnitude, and the opposite pair that have the same MV assignment will have a greater folding angle $\rho_A$ than the opposite pair that have different MV assignments, whose folding angle we'll call $\rho_B$.  If we reparameterize with the tangent of half the folding angles, we have the surprising linear relationship $\tan(\frac{\rho_B}{2} )= p \tan(\frac{\rho_A}{2})$, where $0<p<1$ is a constant, called the {\em folding multiplier} (or sometimes the {\em folding speed}), determined by the plane angles around the vertex.  See \cite{Evans:2015,Tachi2016} for more details and proofs of these folding angle facts.  In the case of the Miura-ori and Chicken Wire patterns, the folding multiplier becomes $p=\cos\theta$, where $\theta$ is the acute angle of the vertices.

\begin{figure}
	\centering
	\includegraphics[width=\linewidth]{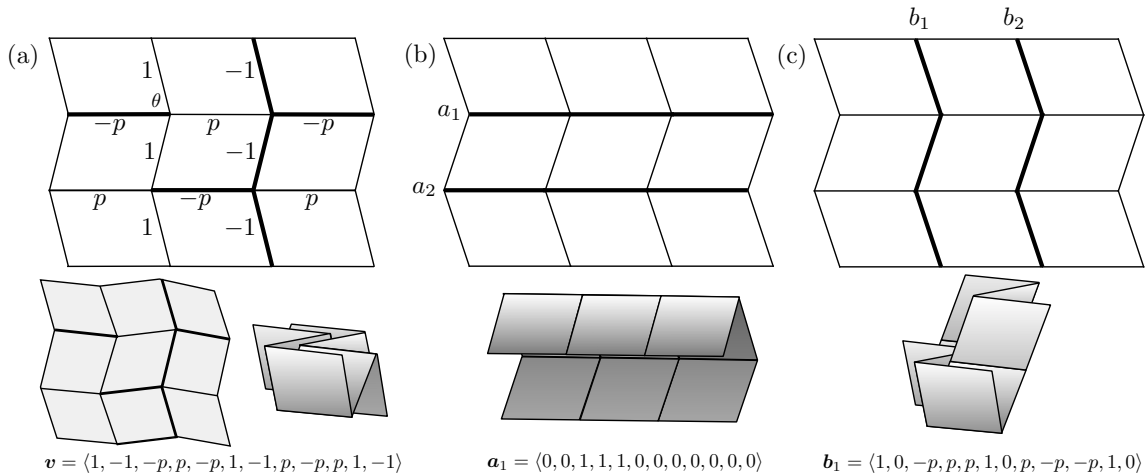}
	\caption{(a) A $3\times 3$ Miura-ori, standard folding, with creases labeled by folding multipliers and written as a vector $\vec v$.  (b) Highlighting the creases $a_1, a_2$.  (c) Highlighting the creases $b_1, b_2$.}
	\label{TT-fig3}
\end{figure}

We consider a $3\times 3$ section of either a Miura-ori (as in Figure~\ref{TT-fig3}(a)) or a Chicken Wire pattern.  To determine its potential for self-foldability, we consider a point $x$ in the configuration space $S$ of the crease pattern and a tangent vector $\vec v$ at $x$ in the direction of a  rigidly folding $\vecrho(t)$ passing through $x$ at $\vecrho(t_0)$ (thus ${\vec v}=\vecrho'(t_0)$).  We may compute the dimension $m$ of the tangent space $T_{\vec v}S$.  By Theorem~\ref{TT-thm3}, we have $m=2+2=4$ for our $3\times 3$ Miura-ori/Chicken Wire tiling.

Since these crease patterns have 12 creases, their configuration space will be subsets of $\mathbb{R}^{12}$.  We coordinatize the folding modes by letting the coordinates of a vector ${\vec v}\in\mathbb{R}^{12}$ be the folding angles of the creases, reading from left-to-right and top-to-bottom, as in Figure~\ref{TT-fig3}.  
We then can describe a basis for $T_{\vec v}S$ at the unfolded state of the Miura-ori:
\begin{align*}
    {\vec a}_1 & = \langle 0,0,1,1,1,0,0,0,0,0,0,0\rangle \\
    {\vec a}_2 & = \langle 0,0,0,0,0,0,0,1,1,1,0,0\rangle \\
    {\vec b}_1 & = \langle 1,0,-p,p,p,1,0,p,-p,-p,1,0\rangle \\
    {\vec b}_2 & = \langle 0,1,-p,-p,p,0,1,p,p,-p,0,1\rangle.
\end{align*}
Note that positive (negative) folding multipliers correspond to valley (mountain) creases.  If we let ${\vec v}$ be the folding mode the unfolded paper to the standard folded state of the Miura-ori (as shown in Figure~\ref{TT-fig3}(a)), then we have that
$${\vec v} = \langle 1, -1, -p, p, -p,1,-1,p,-p,p,1,-1\rangle =  -p{\vec a}_1 + p{\vec a}_2+   {\vec b}_1 - {\vec b}_2.$$
This means that $\vec v$ is within the valid tangents $\overline{V_{x}}$  surrounding $\vec v$, so by Theorem~\ref{TT-cor2}, there is no driving force making rigid folding uniquely self-foldable.  The case of the Chicken Wire crease pattern follows similarly.

\end{proof}


\begin{proof}[Proof of (c) and (d)]
For these proofs we need to use a result from \cite{rigid-foldability-vert:2016}.  For this we define a {\em bird's foot} to be a collection of three creases $c_0, c_1, c_2$ meeting at a vertex with the same MV assignment and $0<\angle(c_i,c_{i+1})<\pi$ for $i=0,1,2$ (mod 3), together with a fourth $c_3$ crease meeting the same vertex but with opposite MV parity to that of $c_0, c_1, c_2$.  Then the following is a special case of the Main Theorem from \cite{rigid-foldability-vert:2016}.

\begin{theorem}\label{TT-thm5}
A degree-4 vertex is rigidly foldable by a finite amount if and only if it is a bird's foot.
\end{theorem}

This result implies a few things.  First, it gives us that any degree-4, rigidly folding vertex must be made of three mountains and one valley, or vice-versa.  That is, exactly one of the four creases will be have opposite MV parity from the others, and we call this crease the {\em different crease} at the vertex.  
Second, none of the sector angles of a degree-4 rigidly folding vertex can be $\geq \pi$.  Third, the sector angles that border the different crease must sum to less that $\pi$.


Consider the two pairs of opposite creases at a degree-4 vertex. 
If the vertex is not mirror-symmetric, (i.e., is not in a Miura or Chicken Wire crease pattern), a pair of opposite creases form non-straight angle, so they divide the plane into a convex side and a concave side with respect to the pair of creases. 
We consider this partition for both pairs of opposite creases. 
The intersection of the concave sides forms a sector, see Figure~\ref{TT-fig3.1}(a), whose incident creases $c_1$ and $c_4$ cannot have the same MV parity (since one of them must be the different crease at the vertex).
Therefore, the opposite angle, i.e., the intersection of the convex sides, see Figure~\ref{TT-fig3.1}(b), have incident creases $c_2$ and $c_3$ that need to have the same MV assignment.  In this way we see that the vertex has two possible modes, since there are only two places (crease $c_1$ or $c_4$ in Figure~\ref{TT-fig3.1}) for the different crease at the vertex.

\begin{figure}
	\centering
	\includegraphics[scale=.7]{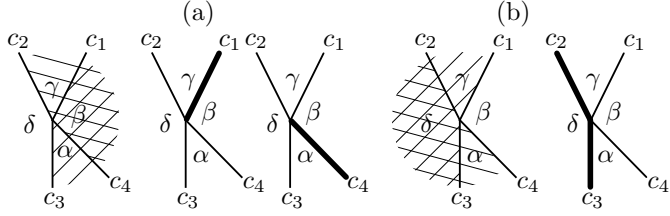}
	\caption{(a) Intersection of the convex sides. (b) The concave sides intersection creases cannot have the same MV parity.}
	\label{TT-fig3.1}
\end{figure}


Now, in the rotationally-symmetric quadrilateral tessellation for case (c) and (d) of Theorem~\ref{TT-thm4}, pick a folding mode (say mode 1) for one of the vertices, say $v_1$ as in Figure~\ref{TT-fig3.2} with three mountains and one valley.  Then the convex-side creases $c_2$ and $c_3$ will form a zig-zag of mountain creases with the convex-side creases of the neighboring vertices, and this will extend to an infinite mountain polyline $P$ of the tessellation.  

Similarly, the concave-side creases $c_1$ and $c_4$ of $v_1$ will have different MV parity.  Assume that $v_1$ being in mode 1 means $c_1$ is M and $c_4$ is V.  Then if $v_4$ is the other vertex adjacent to $c_4$, by the rotational symmetry of the tiling, $c_4$ will also be a concave-side crease of $v_4$, which means that the other concave-side crease, $c_5$ at $v_4$ will have different MV parity from $c_4$.  Thus $c_5$ is a mountain.  The same argument gives that crease $c_8$ in Figure~\ref{TT-fig3.2} is a valley, and another zig-zag polyline of creases in the tessellation will be determined to alternate Ms and Vs.  

Thus vertex $v_3$ in Figure~\ref{TT-fig3.2} will have the same MV assignment as vertex $v_1$, which means that $v_3$ is also mode 1.

\begin{figure}
	\centering
	\includegraphics[scale=.6]{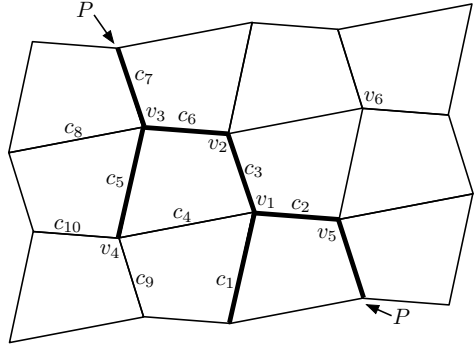}
	\caption{Cascading modes in a rotationally-symmetric monohedral quadrilateral tessellation.}
	\label{TT-fig3.2}
\end{figure}

We then argue that vertex $v_2$ must also be a mode 1 vertex.  This is because the folding angles for the creases $c_3$ and $c_6$ are exactly the folding angles that allow $v_2$ to be mode 1, and if $v_2$ were  mode 2 then a different combination of folding angles for these creases would be needed.  That is, the folding angles equations for these two modes are sufficiently different (in this case where opposite creases at the vertices do not form a straight line) that they do not allow $v_2$ to be in mode 2 while $v_1$ and $v_3$ are in mode 1.\footnote{This is easily verified in the flat-foldable case by examining the actual relative folding angle equations of the creases the actual folding angle equations, as can be seen in \cite{Langtwists,Izmestiev2017,Tachi2016} as well as in the alternate proof below.  The non-flat-foldable equations are much more complicated, but can be seen in \cite{Huff76,Izmestiev2017}.}  Similarly, creases $c_4$ and $c_5$ have folding angles that make vertex $v_4$ be mode 1, and this forces $c_9$ and $c_{10}$ to both be mountains.  (If they were valleys, $v_4$ would be in mode 2.)  

This perpetuates  to the rest of the crease pattern; vertex $v_5$ must also be mode 1 with three Ms and one V, and so must $v_6$ and all the other vertices.  We conclude that the entire tessellation has only two folding modes:  one where all vertices are mode 1, and one where all vertices are mode 2.

These two modes have different tangent vectors in the configuration space at the flat, unfolded state. That is, if $\vec 0$ is the origin, then set of valid tangents $V_{\vec 0}$ consists of four vectors that come in antipodal pairs, one pair for each mode. The pair for one mode will not be contained in the 1-dimensional linear space spanned by the other antipodal, and so by Theorem~\ref{TT-cor2} there exists a driving force making the rigid folding uniquely self-foldable to either mode.

\end{proof}

To help see how the explicit folding angle equations verify the above proof, we give an alternate proof of part (c) of Theorem~\ref{TT-thm4}, which is the flat-foldable monohedral tiling case.  

\begin{proof}[Another proof of (c)]

The vertices made from the monohedral tiling in this case are flat-foldable with adjacent sector angles $\alpha, \beta$ (with $\alpha< 90^\circ$), and thus the other two angles are $\pi-\alpha$ and $\pi-\beta$ to ensure that Kawasaki's Theorem holds at each vertex.  From geometry we know that monohedral tiles that fit these conditions are exactly those inscribable in a circle.

\begin{figure}
	\centering
	\includegraphics[width=\linewidth]{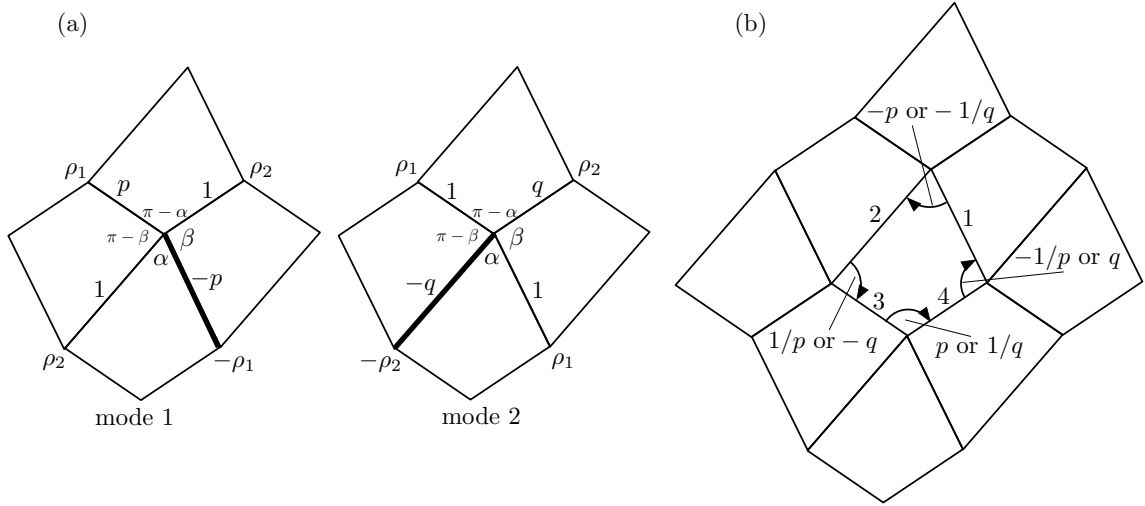}
	\caption{(a) The possible folding multipliers of the creases of a vertex in mode 1 and mode 2.  (b) The possible folding multiplier ratios at each vertex in a quadrilateral of the flat-foldable monohedral tiling.}
	\label{TT-fig4}
\end{figure}

A vertex of such a tiling will, when rigidly folded, have folding angles $\rho_1,\ldots, \rho_4$ that satisfy the following equations (see \cite{Tachi2016} for details):
$$\rho_1=-\rho_3, \rho_2=\rho_4,\mbox{ and }\tan\frac{\rho_1}{2}=p\tan\frac{\rho_2}{2}\mbox{ for mode 1, and}$$
$$\rho_1=\rho_3, \rho_2=-\rho_4,\mbox{ and }\tan\frac{\rho_2}{2}=q\tan\frac{\rho_1}{2}\mbox{ for mode 2,}$$
where mode 1 and mode 2 represent the two ways such a vertex can fold, based on its MV assignment; see Figure~\ref{TT-fig4}(a).  Here $p=\frac{\cos((\alpha+\beta)/2)}{\cos((\alpha-\beta)/2)}$ and $q=-\frac{\sin((\alpha-\beta)/2)}{\sin((\alpha+\beta)/2)}$.  

Now consider a polygonal tile in our origami tessellation, like the one bounded by edges 1, 2, 3, and 4 in Figure~\ref{TT-fig4}(b).  Proceeding counter-clockwise around the polygon, we assign to each vertex a ratio $\mu_i$ of the folding multipliers of the creases before and after the vertex.  For example, the top-most vertex in the polygon in Figure~\ref{TT-fig4}(b) will have $\mu_1=$(the folding multiplier of crease 1)$/$(the folding multiplier of crease 2).  A key observation noted by multiple researchers (such as \cite{Langtwists}) is that in order for such a polygon to rigidly fold, we must have that the product of these folding multiplier ratios around the polygon is 1:
\begin{equation}\label{TTeq1}
\prod \mu_i = 1.
\end{equation}
Carefully noting the possible folding multipliers of the creases in mode 1 or mode 2 of our vertices, as seen in Figure~\ref{TT-fig4}(a), we have two possible values for $\mu_i$ at each vertex in our quadrilateral tile (see Figure~\ref{TT-fig4}(b)).  Thus Equation~\eqref{TTeq1} becomes
\begin{equation}\label{TTeq2}
(-p\mbox{ or }-\frac{1}{q})(\frac{1}{p}\mbox{ or }-q)(p\mbox{ or }\frac{1}{q})(-\frac{1}{p}\mbox{ or }q)=1.
\end{equation}
Equation~\eqref{TTeq2} has only two solutions:
\begin{equation*}
(-\frac{1}{q})(-q)(\frac{1}{q})(q)  = 1\mbox{ and }
(-p)(\frac{1}{p})(p)(-\frac{1}{p})  =1.
\end{equation*}
The MV assignments and crease folding multipliers that result from these solutions are shown in Figure~\ref{TT-fig5}; they perpetuate throughout the entire origami tessellation, and thus prove that there are only two rigidly-foldable modes of this crease pattern.
\begin{figure}
	\centering
	\includegraphics[scale=.5]{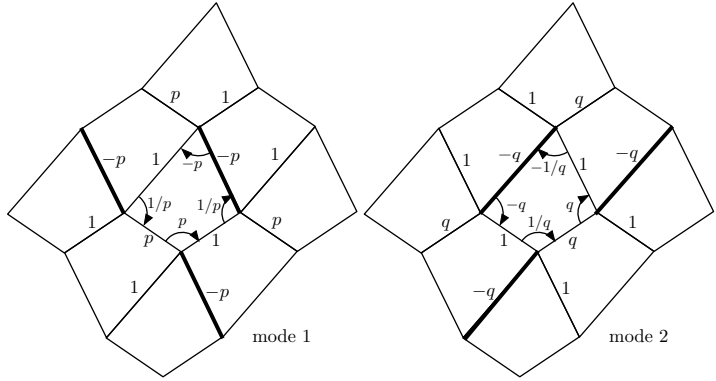}
	\caption{(a) The two folding modes of a flat-foldable, monohedral tiling made from a rotationally-symmetric quadrilateral.  The bold lines are mountains and the non-bold lines are valleys.}
	\label{TT-fig5}
\end{figure}

These two different modes have different tangent vectors at the flat state, so by Theorem~\ref{TT-cor2} there exists a driving force making the rigid folding uniquely self-foldable to either mode.
\end{proof}

\section{Conclusion}
\label{sec:conclusion}

We have proven that monohedral tilings made by quadrilaterals that are rotationally symmetric about the midpoints of their sides are, in most cases, uniquely self-foldable.  That is, if we pick a target folding state reachable by a rigid folding $\vecrho(t)$, there will exist a driving force $\vec f$ that will fold the crease pattern along the proper mode to the target state.  The only such quadrilateral tile that does not fall into this result is the isosceles trapezoid tile, which gives the non-uniquely self-foldable Chicken Wire crease pattern.  The Miura-ori tiling is not rotationally-symmetric, but it fails to be uniquely self-foldable as well and thus we include it in our Theorem~\ref{TT-thm4}.  The proofs of these results amount to the fact that when the quadrilateral tile is generic (in that it lacks the symmetry of parallelograms or isosceles trapezoids), the resulting tessellation will be a rigid origami with only two folding modes, and fewer folding modes imply easier self-foldability.  The Miura-ori and Chicken Wire patterns, however, have more symmetry and thus more folding modes, making them impossible to self-fold according to our definition.  

We should point out that our definition of self-foldability is a relatively abstract model in that it does not take material stiffness or bending energy into account.  Recent work by physicists are providing insight into the practical side of self-foldability models like ours.   In \cite{Stern2018} the authors compare folding from the unfolded state to finding the ground state in a glassy energy landscape, and in \cite{Chen2018} a theory for second-order rigidity of triangulated crease patterns is derived.  Both of these new papers prove the exponential growth of the number of branches at the origin in the configuration space of a rigidly-foldable crease pattern as the number of vertices increases.  But as we have seen in this paper, there are rigidly-foldable origami tessellaitons with very few folding modes (and thus few configuration space branches at the origin).  The rotationally-symmetric quadrilateral tessellations presented here would make good candidates for testing how reliably a crease pattern can self-fold to a target state.  

\bibliographystyle{osmebibstyle}
\bibliography{7OSME-monohedral}


\end{document}